\documentclass[leqno]{amsart}
\usepackage{amsmath}
\usepackage{amssymb}
\usepackage{amsthm}
\usepackage{enumerate}
\usepackage[mathscr]{eucal}
\theoremstyle{plain}
\newtheorem{theorem}{Theorem}[section]
\newtheorem{lemma}[theorem]{Lemma}

\theoremstyle{definition}

\newtheorem{remark}[theorem]{Remark}

\newtheorem{example}[theorem]{Example}

\theoremstyle{remark}

\begin{document}

\title [Zeros of a polynomial and numerical radius inequalities]{Estimations of zeros of a polynomial using numerical radius inequalities  } 

\author{Pintu Bhunia, Santanu Bag, Raj Kumar Nayak and Kallol Paul }

\address{(Bhunia) Department of Mathematics, Jadavpur University, Kolkata 700032, India}
\email{pintubhunia5206@gmail.com}

\address{(Bag) Department of Mathematics, Vivekananda College For Women, Barisha, Kolkata 700008, India}
\email{santanumath84@gmail.com}

\address{(Nayak) Department of Mathematics, Jadavpur University, Kolkata 700032, India}
\email{rajkumarju51@gmail.com}

\address{(Paul) Department of Mathematics, Jadavpur University, Kolkata 700032, India}
\email{kalloldada@gmail.com}

\thanks{First and third author would like to thank UGC, Govt. of India for the financial support in the form of JRF. Prof. Kallol Paul would like to thank RUSA 2.0, Jadavpur University for the partial support.}



\maketitle
\begin{abstract}
We present  new bounds for the numerical radius of bounded linear operators and $2\times 2$ operator matrices. We apply upper bounds for the numerical radius to the Frobenius companion matrix of a complex monic polynomial to obtain  new estimations for zeros of that polynomial. We also show with numerical examples that our new estimations improve on the existing estimations.
\end{abstract}

\section{Introduction}
The purpose of the present article is to present a general method to estimate the zeros of a monic polynomial. The estimation for the zeros of a polynomial have important applications in many areas of sciences such as signal processing, control theory, communication theory, coding theory and cryptography etc. To find the exact zeros of a polynomail of higher order is very difficult and there is no standard method as such.  For this reason, the  estimation of the disk containing all the zeros of a polynomial is an important area of research. Over the years many mathematicians have developed various tools to estimate the disk that contains all the zeros. We  use numerical radius inequalties of  Frobenius companion matrix associated with a given polynomial to find a disk of smaller radius  that contains all the zeros of the polynomial.  This is the time to  introduce some notations and terminologies to be  used in this article.\\

Let $\mathbb{H}_1$, $\mathbb{H}_2$ be two complex Hilbert spaces with usual inner product $\langle.,.\rangle$ and $B(\mathbb{H}_1,\mathbb{H}_2)$ denote the set of all bounded linear operators from $\mathbb{H}_1$ into $\mathbb{H}_2$. If $\mathbb{H}_1=\mathbb{H}_2=\mathbb{H}$ then we write $B(\mathbb{H}_1,\mathbb{H}_2)=B(\mathbb{H}).$ For $T\in B(\mathbb{H})$, the operator norm $\|T\|$ of $T$ is defined as : \[\|T\|=\sup \left \{\|Tx\|:x\in \mathbb{H}, \|x\|=1 \right\}.\]
For $T\in B(\mathbb{H})$, the numerical range $W(T)$, numerical radius $w(T)$ and Crawford number $m(T)$ of $T$ are defined as: 
\begin{eqnarray*} 
W(T)&=&\left\{ \langle Tx,x\rangle : x\in \mathbb{H}, \|x\|=1 \right\},\\
w(T)&=& \sup \left\{ |\mu| : \mu \in W(T) \right\},\\
m(T)&=& \inf \left\{ |\mu| : \mu \in W(T) \right\}.
\end{eqnarray*}
It is easy to verify that $w(T)$ is a norm on $B(\mathbb{H})$ and equivalent to the operator norm satisfying the following inequality 
\[\frac{1}{2}\|T\|\leq w(T)\leq \|T\|.\]
Observe that spectrum $\sigma(T)$ of $T$ is contained in the closure of the numerical range $W(T)$ of $T$, so the spectral radius $r(T)$ of $T$ always satisfies $r(T)\leq w(T).$

Let us consider a monic polynomial of degree $n$, $p(z)=z^n+a_{n-1}z^{n-1}+a_{n-2}z^{n-2}+\ldots+a_1z+a_0,$ where  the coefficients $a_i \in \mathbb{C}$ for  $i=0,1,\ldots,n-1$. The Frobenius companion matrix $C(p)$, associated with polynomial $p(z),$ is given by
$$C(p)=\left(\begin{array}{ccccc}
-a_{n-1}&-a_{n-2}&\ldots&-a_1&-a_0\\
1&0&\ldots&0&0\\
0&1&\ldots&0&0\\
\vdots&\vdots& &\vdots&\vdots\\
0&0&\ldots&1&0\\
\end{array}\right)_{n\times n}.$$
It is easy to verify that all the eigenvalues of $C(p)$ are exactly the zeros of the polynomial $p(z).$ Considering  $C(p)$ as a bounded linear operator on $\mathbb{C}^n,$ we get  $r(C(p))\leq w(C(p))$ and so  $|\lambda|\leq w(C(p)),$  where $\lambda$ is a zero of $p(z).$ If $R$ is radius of a disk with center at the origin that contains all the zeros of the polynomial, then $w(C(p))$ is one such $R$.
Over the years various mathematicians have estimated radius $R$  using various technique. Few of them are listed in below. \\ \\
(1) Abdurakhmanov \cite{A} proved that
  \[|\lambda|\leq\frac{1}{2}\left( |a_{n-1}|+\cos\frac{\pi}{n}+\sqrt{\left(|a_{n-1}|-\cos\frac{\pi}{n}\right)^2+\left(1+\sqrt{\sum^{n-2}_{j=0}|a_j|^2}\right)^2}\right)=R_A.\]
(2) Abu-Omar and Kittaneh \cite{AK1} proved that 
	     \[|\lambda| \leq \sqrt{\frac{1}{4}(|a_{n-1}|^2+\alpha)^2+\alpha+\cos^2\frac{\pi}{n+1}}=R_{AK}, \] where $\alpha=\sqrt{\sum_{j=0}^{n-1}|a_j|^2}$.\\
(3)	Bhunia et. al. \cite{BBP3} proved that
\[|\lambda|\leq |\frac{a_{n-1}}{n}|+\cos\frac{\pi}{n}+\frac{1}{2} [(1+\alpha)^2+4\alpha+4\sqrt{\alpha}(1+\alpha)]^{\frac{1}{4}}=R_{BBP},\]  
where 
\begin{eqnarray*}
\alpha_{r}&=&\sum^{n}_{k=r} {^k}C_r\big(-\frac{a_{n-1}}{n}\big)^{k-r} a_{k},  ~~ r=0,1,\ldots, n-2, a_n=1, {^0}C_0=1, \\
\alpha&=&\sum^{n-2}_{i=0}|\alpha_i|^2.
\end{eqnarray*}		
(4) Cauchy \cite{HJ}	proved that 
		   \[ |\lambda|\leq 1+\max \{ |a_0|, |a_1|, \ldots, |a_{n-1}|\}=R_C.\]
(5) Carmichael and Mason \cite{HJ} proved that 
		\[ |\lambda|\leq (1+|a_0|^2+|a_1|^2+\ldots+|a_{n-1}|^2)^{\frac{1}{2}}=R_{CM}.\]
(6) Fujii and Kubo \cite{FK} proved that 
			\[|\lambda|\leq \cos\frac{\pi}{n+1}+\frac{1}{2}\big[\big(\sum_{j=0}^{n-1}|a_j|^2\big)^{\frac{1}{2}}+|a_{n-1}|\big]=R_{FK}.\]
(7) Kittaneh \cite{K2} proved that 
   \[|\lambda|\leq\frac{1}{2}\left( |a_{n-1}|+1+\sqrt{\left(|a_{n-1}|-1\right)^2+4\sqrt{\sum^{n-2}_{j=0}|a_j|^2}}\right)=R_{K_1}.\]
(8) Kittaneh \cite{K2} proved that 
   \[|\lambda|\leq\frac{1}{2}\left( |a_{n-1}|+\cos\frac{\pi}{n}+\sqrt{\left(|a_{n-1}|-\cos\frac{\pi}{n}\right)^2+(|a_{n-2}|+1)^2+\sum^{n-3}_{j=0}|a_j|^2}\right)=R_{K_2}.\]
(9)	Montel \cite{HJ}	proved that
		\[ |\lambda|\leq \max \left\{1, \sum^{n-1}_{r=o}|a_r|\right \}=R_M.\]
			
In this article, we obtain some upper bounds for the numerical radius of bounded linear operators and operator matrices.  Using these  bounds  and the bounds obtained in \cite{BBP,BBP3,BBP1,BBP2,BPN} we obtain  bounds for the radius of the disk with centre at origin  that contains all the zeros of a complex monic polynomial. Also we show with numerical examples that these bounds obtained here  improve on the existing bounds.

\section{\textbf{Upper bounds for numerical radius}}
In this section we  obtain bounds for the numerical radius of operators which will be used to estimate zeros of polynomial in the next section.  We need the following numerical radius equality \cite{Y}.
\begin{lemma}\label{lem-1b}
Let $T\in B(\mathbb{H})$ and $H_{\theta}=\textit{Re}(e^{i\theta}T)$, where $\theta \in \mathbb{R}$. Then \[w(T)=\sup_{\theta\in \mathbb{R}}\|H_{\theta}\|.\]
\end{lemma}
Using the  Lemma we prove the following theorem.
\begin{theorem}\label{theorem-2.1}
Let $T\in B(\mathbb{H})$. Then \[w^2(T) \leq w(T^2)+\min\left\{\|\textit{Re}(T)\|^2,\|\textit{Im}(T)\|^2\right\}.\]
\end{theorem}

\begin{proof}
Let $H_{\theta}=\textit{Re}(e^{i\theta}T)$, where $\theta \in \mathbb{R}$. Then
\begin{eqnarray*}
	4H^2_{\theta}&=& e^{2i\theta}T^2+  e^{-2i\theta}{T^{*}}^2+TT^{*}+T^*T\\
	\Rightarrow 4H^2_{\theta}&=& (e^{2i\theta}-1)T^2+  (e^{-2i\theta}-1){T^{*}}^2+T^2+{T^{*}}^2+TT^{*}+T^*T\\
	\Rightarrow H^2_{\theta}&=& \frac{1}{2}\textit{Re}\left\{(e^{2i\theta}-1)T^2 \right\}+ (\textit{Re}(T))^2\\
	\Rightarrow H^2_{\theta}&=& \sin\theta\textit{Re}\left\{(e^{i(\theta+\frac{\pi}{2})}T^2 \right\}+ (\textit{Re}(T))^2\\
	\Rightarrow \|H_{\theta}\|^2&\leq& \|\textit{Re}\left\{(e^{i(\theta+\frac{\pi}{2})}T^2 \right\}\|+ \|\textit{Re}(T)\|^2\\
	&\leq& w(T^2)+\|\textit{Re}(T)\|^2, ~~\mbox{using Lemma \ref{lem-1b}}.
\end{eqnarray*}
Taking supremum over $\theta \in \mathbb{R}$ and then using Lemma \ref{lem-1b} we get 
 \[w^2(T) \leq w(T^2)+\|\textit{Re}(T)\|^2.\] 
Applying similar argument we can prove that
\[w^2(T) \leq w(T^2)+\|\textit{Im}(T)\|^2.\] 
This completes the proof of the theorem.
\end{proof}	

\begin{remark}
It follows from Theorem \ref{theorem-2.1} that if $T^2=0$ then $w(T)\leq \|\textit{Re}(T)\|$ and $w(T)\leq \|\textit{Im}(T)\|.$ From  \cite[Th. 3.3]{BBP} it follows  that for any $T\in B(\mathbb{H})$,  
$$\|\textit{Re}(T)\|^2+m^2(\textit{Im}(T))\leq w^2(T),$$ $$\|\textit{Im}(T)\|^2+m^2(\textit{Re}(T))\leq w^2(T).$$ If $T^2=0$ then we get  $w(T)= \|\textit{Re}(T)\|= \|\textit{Im}(T)\|$  and $m(\textit{Re}(T))=m(\textit{Im}(T))=0.$
Also we have from Theorem \ref{theorem-2.1} and \cite[Th. 3.3]{BBP} that for any $T\in B(\mathbb{H}),$  $m(\textit{Re}(T))\leq \sqrt{w(T^2)}$ and $m(\textit{Im}(T))\leq\sqrt{w(T^2)}$.
\end{remark}

Next we obtain an upper bound for the numerical radius of $2 \times 2$ operator matrices.

\begin{theorem}\label{theorem-2.2}
Let  $T=\left(\begin{array}{cc}
A&B \\
C&D
\end{array}\right)$, where $A\in B(\mathbb{H}_1),B\in  B(\mathbb{H}_2,\mathbb{H}_1) ,C\in  B(\mathbb{H}_1,\mathbb{H}_2),D\in  B(\mathbb{H}_2)$.
 Then 
 $$ w(T) \leq \frac{1}{2} \left[ w(A) + w(D) + \sqrt{(w(A)-w(D))^{2}+\|B\|^{2}+\|C\|^{2}+2w(CB) } \right]. $$
 
\end{theorem}
\begin{proof}
Abu-Omar and Kittaneh in \cite[Cor. 2]{AK} proved that $$ w(T) \leq \frac{1}{2} \left[ w(A) + w(D) + \sqrt{(w(A)-w(D))^{2}+4w^2(T_0) } \right]. $$ where $T_0=\left(\begin{array}{cc}
O&B \\
C&O
\end{array}\right).$ \\
We proved in \cite[Th. 2.5]{BBP3} that \[w^4\left(\begin{array}{cc}
    O&B \\
    C&O
\end{array}\right)\leq \frac{1}{16}\|S\|^2+\frac{1}{4}w^2(CB)+\frac{1}{8}w(CBS+SCB),\]   where $S=|B|^2+|C^{*}|^2$. Our required bound follows from these above two bounds, using the facts that $w(CBS+SCB)\leq 2w(CB)\|S\|$, ( see \cite[Remark 2.15]{BPN}) and $\|S\|\leq \|B\|^2+\|C\|^2$.
\end{proof}
\begin{remark}
Paul and Bag in \cite[Th. 2.1, (i)]{PB} proved that
\[w(T)\leq \frac{1}{2} \Big[ w(A) + w(D) + \sqrt{(w(A)-w(D))^{2}+ (\|B\|+\|C\|)^2} \Big].\]
Clearly it is weaker than the inequality obtained in Theorem \ref{theorem-2.2}.
\end{remark}

Next we give  another upper bound for the numerical radius of $2 \times 2$ operator matrices.

\begin{theorem} \label{theorem-2.3}
Let  $T=\left(\begin{array}{cc}
A&B \\
C&D
\end{array}\right)$, where $A\in  B(\mathbb{H}_1),B\in  B(\mathbb{H}_2,\mathbb{H}_1) ,C\in  B(\mathbb{H}_1,\mathbb{H}_2),D\in  B(\mathbb{H}_2).$ Then 
$$ w(T) \leq \frac{1}{2} \left[ w(A) + w(D) + \sqrt{(w(A)-w(D))^{2}+ 4\alpha_1^2} \right], $$ where 
$$ \alpha_1 = \left[\frac{1}{8} \max\left\{\left(\|B\|^2 + \|C\|^2\right)^2 + 4 w^2\left(BC\right), \left(\|B\|^2 + \|C\|^2\right)^2 + 4 w^2\left(CB\right)\right\}\right]^{\frac{1}{4}}.$$	
\end{theorem}	
\begin{proof}
This inequality follows from  the two inequalities proved in \cite[Cor. 2]{AK}  and  \cite[Th. 2.7]{BBP2},  respectively stated below:
 $$ w(T) \leq \frac{1}{2} \left[ w(A) + w(D) + \sqrt{(w(A)-w(D))^{2}+4w^2(T_0) } \right], T_0=\left(\begin{array}{cc}
O&B \\
C&O
\end{array}\right).$$ \\
and 
\[w^{4}\left(\begin{array}{cc}
	O&B \\
	C&O
\end{array}\right) \leq \frac{1}{8}\max \Big\{\| BB^{*}+C^{*}C\|^2+4w^2(BC), \| B^{*}B+CC^{*}\|^2  + 4w^2(CB)\Big\}.\]

\end{proof}

Next we obtain some upper bounds for the numerical radius of a bounded linear operator defined on  complex Hilbert space $\mathbb{H}$. We need  the  Aluthge transform of an operator $T$.  
For $T\in B(\mathbb{H})$, the Aluthge  transform \cite{JKP} of $T$, denoted as $\widetilde{T}$, is defined as \[\widetilde{T}=|T|^{\frac{1}{2}}U|T|^{\frac{1}{2}}\] where $|T|=(T^{*}T)^{\frac{1}{2}}$ and $U$ is the partial isometry associated with the polar decomposition of $T$ and so $T=U|T|, \ker T=\ker U.$ It follows easily from the definition of  $\widetilde{T}$ that  $\|\widetilde{T}\| \leq \|T\|$ and  $r(\widetilde{T})= r(T)$, also  $w(\widetilde{T}) \leq w(T)$ (see \cite{JKP}).

For definition and more information of Aluthge transform we refer the reader to \cite{Y} and references therein.

\begin{theorem} \label{theorem-2.4}
Let $T\in  B(\mathbb{H}).$ Then $$w^2(T) \leq \frac{1}{4}\|T\|\|T^2\|^\frac{1}{2} + \frac{1}{4}\|T^2\| + \frac{1}{2}\|T\|^2.$$
\end{theorem}

\begin{proof}
The proof follows from the inequality \cite[Th. 2.6]{BBP1} 
\[w^2(T)\leq \frac{1}{4}w(\widetilde{T}^2)+\frac{1}{4}\|T\|\|\widetilde{T}\|+\frac{1}{4}\|T^{*}T+TT^{*}\|\] 
and  following simple inequalities
$w(\widetilde{T}^2)\leq \|T\|^2$, $\|\widetilde{T}\|\leq \|T^2\|^{\frac{1}{2}}, \|T^*T+TT^*\|\leq \|T^2\|+\|T\|^2.$
\end{proof}

We next obtain the following inequality, which can be proved using  the inequality \cite[Th. 2.3]{BBP1},  $w^2(T)\leq \frac{1}{2}\|T\|\|\widetilde{T}\|+\frac{1}{4}\|T^{*}T+TT^{*}\| $ and noting that $\|\widetilde{T}\|\leq \|T^2\|^{\frac{1}{2}}$ and $\|T^*T+TT^*\|\leq \|T^2\|+\|T\|^2.$ 

\begin{theorem} \label{theorem-2.5}
Let $T\in  B(\mathbb{H}).$ Then $$ w^2(T) \leq \frac{1}{2}\|T\|\|T^2\|^\frac{1}{2} + \frac{1}{4}\|T^2\|+ \frac{1}{4}\|T\|^2.$$
\end{theorem}
We end this section with the following inequality whcih can be proved using the inequality  \cite[Th. 2.1]{BBP},  $ w^{4}(T)\leq \frac{1}{4}w^2(T^2)+\frac{1}{8}w(T^2P+PT^2)+\frac{1}{16}\|P\|^2,$  where $P=T^{*}T+TT^{*}$ and $w(T^2)\leq \|T^2\|$, $w(T^2P+PT^2)\leq 2w(T^2)\|P\|$ ( \cite[Remark 2.15]{BPN}), $\|P\|\leq \|T^2\|+\|T\|^2$.

\begin{theorem} \label{theorem-2.6}
Let $T\in  B(\mathbb{H}).$ Then $$ w^2(T) \leq \frac{3}{4}\|T^2\| + \frac{1}{4}\|T\|^2.$$
\end{theorem}



\section{\textbf{Estimation for zeros of a polynomial}}

Consider a monic polynomial of degree $n$, $p(z)=z^n+a_{n-1}z^{n-1}+a_{n-2}z^{n-2}+\ldots+a_1z+a_0,$ where  the coefficients $a_i \in \mathbb{C}$ for  $i=0,1,\ldots,n-1$.  Let $R$ denote radius of a disk with center at origin that contains all the zeros of the polynomial $p(z)$. If $\lambda $ is a zero of the polynomial $p(z)$,  equivalently,  
if $\lambda $ is an eigen value of the Frobenius companion matrix $C(p) $ (as described in the introduction), then $ \mid \lambda \mid \leq R.$  Our goal in this section is to obtain smaller possible values of $R.$ To do so 
we need  the following two well known results on numerical radius equality.

\begin{lemma}\cite[pp. 8-9]{GR}\label{lem-1}\\
Let $L_n=\left(\begin{array}{ccccc}
0&0&\ldots&0&0\\
1&0&\ldots&0&0\\
0&1&\ldots&0&0\\
\vdots&\vdots& &\vdots&\vdots\\
0&0&\ldots&1&0\\
\end{array}\right)$ be an ${n\times n}$ matrix. \\Then $w(L_n)=\cos \frac{\pi}{n+1}.$
\end{lemma}

\begin{lemma}\cite{FK}\label{lem-2}\\
Let $x_i\in \mathbb{C}$ for each $i=1,2,\ldots,n.$ Then\\ 
 $$w\left(\begin{array}{cccc}
x_1&x_2&\ldots&x_n\\
0&0&\ldots&0\\
0&&\ldots&0\\
\vdots&\vdots& &\vdots\\
0&0&\ldots&0\\
\end{array}\right)=\frac{1}{2}\left( |x_1|+ \sqrt{\sum \limits_{r=1}^{n}|x_r|^2} \right).$$
\end{lemma}

\noindent Using  above two Lemmas  \ref{lem-1} and  \ref{lem-2}, we obtain some new bounds for zeros of the polynomial $p(z)$. First using Theorem \ref{theorem-2.2}, we prove the following theorem.

\begin{theorem} \label{theorem-3.1}
Let $ \lambda $ be a zero of $p(z)$.Then 
$$ |\lambda| \leq w(C(p))\leq \frac{1}{2}\left[|a_{n-1}| + cos\frac{\pi}{n} + \sqrt{(|a_{n-1}|-cos\frac{\pi}{n})^2 +\sum \limits_{r=0}^{n-2}|a_r|^2 + 1 + \alpha} \right]=R_1, $$ where $ \alpha = |a_{n-2}|+ \sqrt{\sum \limits_{r=0}^{n-2}|a_r|^2}. $
\end{theorem}

\begin{proof}
Let  $C(p)=\left(\begin{array}{cc}
	A&B \\
	C&D
\end{array}\right)$, where $ A = (a_{n-1})_{1\times 1}$, $C=\left(\begin{array}{c}
1 \\
0\\
\vdots\\
0
\end{array}\right)_{n-1 \times 1} $,\\  $B=(-a_{n-2} -a_{n-3} \ldots -a_1 -a_0 )_{1\times n-1}$  and $D=L_{n-1}.$ Then using Lemma \ref{lem-1} and Lemma \ref{lem-2} in Theorem \ref{theorem-2.2} we get
$$w(C(p))\leq \frac{1}{2}\left[|a_{n-1}| + cos\frac{\pi}{n} + \sqrt{(|a_{n-1}|-cos\frac{\pi}{n})^2 +\sum \limits_{r=0}^{n-2}|a_r|^2 + 1 + \alpha} \right], $$  
where  $\alpha = |a_{n-2}|+ \sqrt{\sum \limits_{r=0}^{n-2}|a_r|^2}. $ This completes the proof.
\end{proof}

Next using Theorem \ref{theorem-2.3}, we prove the following theorem.

\begin{theorem} \label{theorem-3.2}
Let $\lambda$ be a zero of $p(z)$. Then 
$$|\lambda| \leq w(C(p))\leq  \frac{1}{2}\left[|a_{n-1}| + \cos\frac{\pi}{n} + \sqrt{\left(|a_{n-1}| - \cos\frac{\pi}{n}\right)^2 + 4\alpha^2}\right]=R_2,$$	where 
\begin{eqnarray*}
\alpha &=& \left[\frac{1}{8}\max\left\{\left(\beta +1\right)^2 + 4|a_{n-2}|^2, \left(\beta +1\right)^2 +  \delta^2 \right\}\right]^{\frac{1}{4}}, \\
\beta &=& \sum \limits_{r=0}^{n-2}|a_r|^2, \\
\delta &=&  |a_{n-2}|+ \sqrt{\sum \limits_{r=0}^{n-2}|a_r|^2}.
\end{eqnarray*}	
\end{theorem}

\begin{proof}
We consider $C(p)=\left(\begin{array}{cc}
	A&B \\
	C&D
\end{array}\right)$ where $A,B,C$ and $D$ are same as in the proof of Theorem \ref{theorem-3.1}. Then using Lemma \ref{lem-1} and Lemma \ref{lem-2} in Theorem \ref{theorem-2.3} we have the desired bound.
\end{proof}

In the following example we show with a numerical example that our estimations in Theorem \ref{theorem-3.1} and Theorem \ref{theorem-3.2} are better than the existing estimations.

\begin{example}\label{rem-1}
We consider a polynomial $p(z)=z^5+4z^4+z^3+z^2+z+1.$ Then the upper bounds for the zeros of this polynomial $p(z)$ estimated by different mathematicians are as shown in the following table.
\begin{center}
  \begin{tabular}{ |c|c| } 
\hline
$R_{FK}$ & $5.1020$ \\
\hline
 $R_{K_1}$  & $4.5615$\\
 \hline
 $R_{AK}$ &  $4.8131$ \\ 
\hline 
 $R_{A}$ & $4.5943$  \\
\hline 
$R_{BBP}$ & $7.2809.$\\
\hline
\end{tabular}
\end{center}
But our bounds obtain in Theorem \ref{theorem-3.1} give $R_1=4.5365$ and Theorem \ref{theorem-3.2} give $R_2=4.5509$. Therefore for this polynomial $p(z)$, our obtain bounds in Theorem \ref{theorem-3.1} and Theorem \ref{theorem-3.2} are better than the above mentioned bounds.
\end{example}

We next obtain an estimation of radius $R$ and for that we need the following numerical radius inequality \cite[Cor. 3.6]{BPN}.

\begin{lemma} \label{lem-4}
Let  $T=\left(\begin{array}{cc}
A&B \\
C&D
\end{array}\right)$, where $A\in B(\mathbb{H}_1),B\in  B(\mathbb{H}_2,\mathbb{H}_1) ,C\in  B(\mathbb{H}_1,\mathbb{H}_2),D\in  B(\mathbb{H}_2)$. Then 
$$w(T) \leq \sqrt { w^2(A) + \frac{1}{2}\|B\|\left(w(A)+\frac{1}{2}\|B\|\right)} +  \sqrt { w^2(D) + \frac{1}{2}\|C\|\left(w(D)+\frac{1}{2}\|C\|\right)}.$$
\end{lemma}

\begin{theorem} \label{theorem-3.3}
Let $\lambda$ be a zero of $p(z).$ Then
$$ |\lambda| \leq w(C(p))\leq  \sqrt{|a_{n-1}|^2 + \frac{1}{2}\alpha\left(|a_{n-1}| + \frac{1}{2} \alpha\right)} + \sqrt{\cos^2\frac{\pi}{n} + \frac{1}{2}\left(\cos\frac{\pi}{n} + \frac{1}{2}\right)}=R_3,$$ 
where  $\alpha = \sqrt{\sum \limits_{r=0}^{n-2}|a_r|^2}.$
\end{theorem}

\begin{proof}
The proof follows from  using Lemma \ref{lem-4}, Lemma \ref{lem-1} and  similar argument as in the proof of Theorem \ref{theorem-3.1}.  
\end{proof}

The next example highlights that the above estimation is better than the existing ones.

\begin{example}\label{rem-2}
We consider a polynomial $p(z)=z^5+z^3+z+2.$ Then the upper bounds for the zeros of this polynomial $p(z)$ estimated by different mathematicians are as shown in the following table.
\begin{center}
 \begin{tabular}{ |c|c| } 
\hline
$R_{C}$ & $3$ \\
\hline
 $R_{M}$  & $4$\\
 \hline
 $R_{CM}$ &  $2.6457$ \\ 
\hline 
\end{tabular}
\end{center}
But our bound in Theorem \ref{theorem-3.3} gives $R_3=2.3688.$ Therefore for this polynomial $p(z)$, the estimation for zeros of polynomial in Theorem \ref{theorem-3.3} is better than  the existing estimations mentioned above.
\end{example}

We need the following Lemma  \cite[Cor. 3.7]{BPN} to prove the next theorem.

\begin{lemma} \label{lem-5}
Let  $T=\left(\begin{array}{cc}
A&B \\
C&D
\end{array}\right)$, where $A\in B(\mathbb{H}_1),B\in  B(\mathbb{H}_2,\mathbb{H}_1) ,C\in  B(\mathbb{H}_1,\mathbb{H}_2),D\in  B(\mathbb{H}_2)$.Then 
\[w(T) \leq \sqrt{2w^2(A) + \frac{1}{2}\left(\|A^*B\| + \|B\|^2\right)} + \sqrt{2w^2(D) + \frac{1}{2}\left(\|D^*C\| + \|C\|^2\right)}.\]
\end{lemma}

\begin{theorem} \label{theorem-3.4}
Let $\lambda$ be a zero of $p(z)$.Then
\[|\lambda| \leq w(C(p))\leq  \sqrt{2|a_{n-1}|^2 + \frac{1}{2}\left(\alpha + \beta^2\right)} + \sqrt{2\cos^2\frac{\pi}{n} +\frac{1}{2}}=R_4,\]
where $\alpha=\sqrt{\sum \limits_{r=0}^{n-2}|a_ra_{n-1}|^2}$ and $\beta = \sqrt{\sum \limits_{r=0}^{n-2}|a_r|^2}.$ 
\end{theorem}
 
\begin{proof}
The proof follows from Lemma \ref{lem-5}, by using Lemma \ref{lem-1} and the similar argument as in the proof of Theorem \ref{theorem-3.1}.  
\end{proof}
As before we provide an example.
\begin{example}\label{rem-3}
We consider a polynomial $p(z)=z^5+z^3+z+5.$ Then the upper bounds for the zeros of this polynomial $p(z)$ estimated by different mathematicians are as shown in the following table.
\begin{center}
  \begin{tabular}{ |c|c| } 
\hline
$R_{C}$ & $6$ \\
\hline
 $R_{M}$  & $7$\\
 \hline
 $R_{CM}$ &  $5.2915$ \\ 
\hline 
\end{tabular}
\end{center}
But our bound in Theorem \ref{theorem-3.4} gives $R_4=5.0192.$ Therefore for this polynomial $p(z)$, the estimation for zeros of polynomial in Theorem \ref{theorem-3.4} is better than the existing estimations mentioned in this example.
\end{example}

We state an upper bound for numerical radius of $2 \times 2$ operator matrices \cite[Cor. 3.4]{KS} and using it we prove our next theorem.

\begin{lemma} \label{lem-6}
Let  $T=\left(\begin{array}{cc}
A&B \\
C&D
\end{array}\right)$, where $A\in  B(\mathbb{H}_1),B\in  B(\mathbb{H}_2,\mathbb{H}_1) ,C\in  B(\mathbb{H}_1,\mathbb{H}_2),D\in  B(\mathbb{H}_2)$ Then
\[w\left(\begin{array}{cc}
A&B \\
C&D
\end{array}\right) \leq \frac{1}{2}\left[w(A) + w(D) + \sqrt{w^2(A) + \|B\|^2} + \sqrt{w^2(D) + \|C\|^2}\right].\]
\end{lemma}

\begin{theorem} \label{theorem-3.5}
Let $\lambda$ be a zero of $p(z).$ Then
\[|\lambda| \leq w(C(p))\leq  \frac{1}{2} \left[|a_{n-1}| + \cos\frac{\pi}{n} + \sqrt{|a_{n-1}|^2 + \alpha} + \sqrt{\cos^2\frac{\pi}{n} +1}\right]=R_5,\]
 where $\alpha =\sqrt{ \sum \limits_{r=0}^{n-2}|a_r|^2}. $
\end{theorem}

\begin{proof}
The proof follows from Lemma \ref{lem-6}, by using  Lemma \ref{lem-1} and the similar argument as in the proof of Theorem \ref{theorem-3.1}.  
\end{proof}
Again we give an example to show that the estimation is better than the existing ones.
\begin{example}
We consider the same polynomial $p(z)$ in Remark \ref{rem-2}, i.e.,  $p(z)=z^5+z^3+z+2.$ Then the upper bounds for the zeros of this polynomial $p(z)$ estimated by different mathematicians are as shown in the following table.
\begin{center}
  \begin{tabular}{ |c|c| } 
\hline
$R_{C}$ & $3$ \\
\hline
 $R_{M}$  & $4$\\
 \hline
 $R_{CM}$ &  $2.6457$\\ 
\hline 
 $R_{FK}$ & $2.0907$ \\
\hline 
$R_{A}$ & $2.1760$\\
\hline
$R_{K_1}$ & $2.1430$\\
\hline
$R_{K_2}$ & $1.9580$\\
\hline
$R_{AK}$ & $2.1678$\\
\hline
\end{tabular}
\end{center}
But our bound in Theorem \ref{theorem-3.5} gives $R_5=1.8301.$ Therefore for this polynomial $p(z)$, the estimation in Theorem \ref{theorem-3.5} is better than all the existing estimations mentioned in this example.
\end{example}

Next we give the following two lemmas which can be found in \cite[pp. 335-336]{K1} and \cite[Th. 2.1]{KS} respectively.

\begin{lemma} \label{lem-8}
$$\|C(p)\| = \left[\frac{1}{2}\left(1 + \sum \limits_{r=0}^{n-1}|a_r|^2 + \sqrt{\left(1+ \sum \limits_{r=0}^{n-1}|a_r|^2 \right)^2 -4|a_0|^2}\right) \right]^\frac{1}{2}.$$
\end{lemma}
 
\begin{lemma} \label{lem-9}
 \[ \|C^2(p)\| \leq \left(1 + \sum \limits_{r=0}^{n-1}\left(|a_r|^2 +|b_r|^2 \right) \right)^\frac{1}{2}, \]
 where $b_r = a_{n-1}a_r-a_{r-1}$ for each $r=0,1,\ldots,n-1$ with $a_{-1} =0.$\\
\end{lemma}

Using Theorem \ref{theorem-2.4}, we now prove the following theorem.

\begin{theorem} \label{theorem-3.6}
Let $\lambda$ be a zero of $p(z).$Then 
$$ |\lambda| \leq w(C(p))\leq  \left(\frac{1}{4}\beta\sqrt{\alpha} + \frac{1}{4}\alpha + \frac{1}{2}\beta^2\right)^\frac{1}{2}=R_6,$$ where  

\begin{eqnarray*}
 \alpha &=& \left[1 + \sum \limits_{r=0}^{n-1}\left(|a_r|^2 +|b_r|^2 \right) \right]^\frac{1}{2},\\ 
 \beta &=& \left[\frac{1}{2}\left(1 + \sum \limits_{r=0}^{n-1}|a_r|^2 + \sqrt{\left(1+ \sum \limits_{r=0}^{n-1}|a_r|^2 \right)^2 -4|a_0|^2}\right) \right]^\frac{1}{2}, \\
b_r &=& a_{n-1}a_r-a_{r-1}~~\mbox{ for each } ~~r=0,1,\ldots,n-1 ~~\mbox{ and}~~ a_{-1} =0.
\end{eqnarray*}
\end{theorem}

\begin{proof}
Taking $ T=C(p)$ in Theorem \ref{theorem-2.4} and using Lemma \ref{lem-8} and Lemma \ref{lem-9}, we get the required bound. 
\end{proof}

Next using Theorem \ref{theorem-2.5}, we prove the following theorem.

\begin{theorem} \label{theorem-3.7}
Let $\lambda$ be a zero of $p(z).$ Then $$ |\lambda| \leq w(C(p))\leq  \left(\frac{1}{2}\beta\sqrt{\alpha} + \frac{1}{4} \alpha + \frac{1}{4}\beta^2\right)^\frac{1}{2}=R_7, $$ where $\alpha$ and $\beta$ are same as in Theorem \ref{theorem-3.6}.
\end{theorem}

\begin{proof}
Taking $ T=C(p)$ in Theorem \ref{theorem-2.5}, and using Lemma \ref{lem-8} and Lemma \ref{lem-9} we get the desired bound. 
\end{proof}

Our last theorem in this section is the following one.

\begin{theorem} \label{theorem-3.8}
Let $\lambda$ be a zero of $p(z).$Then $$ |\lambda| \leq w(C(p))\leq \left[\frac{3}{4}\alpha + \frac{1}{4} \beta^2 \right]^\frac{1}{2}=R_8, $$ where
$\alpha$ and $\beta$ are same as in Theorem \ref{theorem-3.6}.
\end{theorem}

\begin{proof}
Taking $ T=C(p)$ in Theorem \ref{theorem-2.6} and using Lemma \ref{lem-8} and Lemma \ref{lem-9} we get the required bound for zeros of $p(z)$. 
\end{proof}

We illustrate with a numerical example to show that the  bounds for zeros of a  polynomial obtained by us in Theorem \ref{theorem-3.6}, Theorem \ref{theorem-3.7}, Theorem \ref{theorem-3.8} are better than the existing bounds.

\begin{example}
We consider a polynomial $p(z)=z^5+2z^4+z^3+z^2+z+1$. Then the upper bounds for the zeros of this polynomial $p(z)$ estimated by different mathematicians are as shown in the following table.
\end{example}
\begin{center}
  \begin{tabular}{ |c|c| } 
\hline
$R_A$ & 3.0183 \\
\hline
 $R_{CM}$ & 3.0000\\
 \hline
$R_C$ &  3.0000 \\ 
\hline 
 $R_{FK}$ & 3.2802  \\
\hline 
$R_{K_1}$ & 3.0000\\
\hline
$R_{K_2}$ & 2.8552\\
\hline
$R_{AK}$ & 3.0670 \\
\hline
\end{tabular}
\end{center}
But for the polynomial  $p(z)=z^5+2z^4+z^3+z^2+z+1$, we have $R_6=2.7129$, $R_7=2.6086$ and $R_8=2.4437$. This shows that for this example, our bounds obtained in Theorem \ref{theorem-3.6}, Theorem \ref{theorem-3.7} and Theorem \ref{theorem-3.8}  are better than all the estimations mentioned above.

\bibliographystyle{amsplain}

\end{document}